\documentclass[12pt,a4]{article}
\usepackage[applemac]{inputenc}  
\usepackage[english]{babel}
\usepackage{amssymb,amsmath,amsthm,amsfonts,amscd,amstext}
\usepackage{array}
\usepackage[all]{xy}
\usepackage{graphicx}
\newtheorem{lem}{Lemma}[section]
\newtheorem{prop}[lem]{Proposition}
\newtheorem{thm}[lem]{Theorem}

\newtheorem{df}{Definition}[section]
\newtheorem{rem}[lem]{Remark}

\def\C{{\mathbb C}}  

\def\Hom{{\operatorname{Hom}}}
\def\Q{{\mathbb Q}}   
\def\R{{\mathbb R}}    
\def\Z{{\mathbb Z}}    
\def\N{{\mathbb N}}    
\title{The $S$-transform in arbitrary dimensions}
\author{Roland Friedrich and John McKay}
\begin{document}
\maketitle
\begin{abstract}
In this note we report the solution of the problem of defining the $S$-transform in Free Probability Theory in arbitrary dimensions. This is achieved by generalising the theory and embedding it into an algebraic-geometric framework. Finally, we classify the groups arising as distributions of $s$-tuples of non-commutative random variables.
\end{abstract}
\section{Introduction}
The addition and multiplication problem for pairs of free non-commutative random variables was solved by Voiculescu~\cite{V}. He introduced two series, the $S_V$ and the $R_V$ transform, cf.~\cite{VDN}. Subsequently, Speicher~\cite{S} and Nica and Speicher~\cite{NS} extended the $R$-transform to arbitrary dimensions considering the free cumulants instead of the moment series. The essential tool to achieve this, is the insight that free probability theory can be described in terms of non-crossing partitions, as realised and developed by Speicher~\cite{S}. Further, a multiplicative operation, the so-called ``boxed" convolution was introduced and studied. Among the many remarkable properties it has, one should mention that it defines a generally non-commutative group structure on the set of non-commutative power series with complex coefficients.

Earlier, Voiculescu described in the one-dimensional case the group structure on non-commutative probability laws as a projective limit of finite dimensional complex Lie groups~\cite{V}.

Since its introduction in $1987$~\cite{V},  however, the nature of the higher-dimensional $S$-transform has remained an unresolved problem. Additionally, the question of how to go, from the additive to the multiplicative problem, if possible at all, also stayed open, cf.~\cite{S}.

Here, we give the answers to all the questions raised above. The main tool to do so is by generalising the entire theory of free probability theory from its usual $\C$-valued setting to arbitrary commutative unital rings, which by its combinatorial nature is possible, and then to study it by using algebraic-geometric methods. As it turns out, the theory becomes nice and pleasant.

This is a slightly revised version of a talk delivered at the Fields Institute in July 2013. Therefore, it contains only the ``bare bones" and some of the major results. General references for standard results and common notations of Free Probability Theory include~\cite{NS,VDN}. The detailed content, in particular the proofs, and also further extensions can be found in the following (up-coming) publications~\cite{FMcK1,FMcK2}.

\section{The one-dimensional case}
We prove that in the one-dimensional non-commutative case, the situation is analogous to classical probability theory.
\begin{thm}[\cite{FMcK1}]
\label{LOG}
There exists a group isomorphism $\operatorname{LOG}:(\Sigma^{\times}_1, \boxtimes)\rightarrow(\Sigma,\boxplus)$, with inverse $\operatorname{EXP}$, defined by the diagram: 
\[
\begin{xy}
  \xymatrix{
      \Lambda(\C) \ar[d]_{\frac{d}{dz}\log}^{\text{``ghost map"}} & (\Sigma^{\times}_1, \boxtimes_V) \ar[l]_{S_V} \ar[d]^{\operatorname{LOG}} \\
                              {\C[[z]]}\ar[r]^{R^{-1}_V} & (\Sigma,\boxplus_V) 
               }
\end{xy}
\]
\end{thm}
In fact, ${\C[[z]]}$ is a {\bf commutative unital ring}, endowed with the {\bf Hadamard multiplication}:
$$
\label{Hada}
\left(\sum_{n=0}^{\infty}a_n z^n\right)\star\left(\sum_{n=0}^{\infty}b_n z^n\right):=\sum_{n=0}^{\infty}a_n b_n z^n~.
$$

\section{Witt vectors and Free Probability}
The algebraic structure of the distribution of non-commutative random variables is isomorphic to the abelian part of the Witt vectors, as we showed in~\cite{FMcK1}.  
\begin{thm}[\cite{FMcK1}]
\begin{itemize}
\item The set $\Sigma^{\times}_1$ of probability distributions with mean $1$, carries the structure of a { commutative unital ring},  with binary operations $(\boxtimes_V,\circledast)$, such that $(\Sigma^{\times}_1,\boxtimes_V,\circledast)$ is isomorphic to the ring $(\Lambda(\C),\cdot,\ast,1, 1-z)$. The multiplication $\circledast$ in $\Sigma_1^{\times}$ satisfies 
$$
\mu_1\circledast\mu_2=S_V^{-1}(S_V(\mu_1)\ast S_V(\mu_2))\quad\text{for $\mu_1,\mu_2\in\Sigma_1^{\times}$}.
$$
\item There exists an unital ring structure on $\Sigma$, with binary operations $(\boxplus_V,\boxdot)$, where $\boxdot$ denotes the multiplication. This ring is isomorphic to the ring $(\C[[z]],+,\star,\underline{0},\underline{1})$.
\item The rings $(\Sigma^{\times}_1,\boxtimes_V,\circledast)$ and $(\Sigma,\boxplus_V,\boxdot)$ are isomorphic, with the isomorphisms given by $\operatorname{LOG}$ and $\operatorname{EXP}$, respectively.
\end{itemize}
\end{thm}

\section{Partial ring structure on one-dimensional probability measures}
As a consequence we derive a partial ring structure for the one-dimensional probability distributions have a partial ring structure, i.e. certain types of probability measures behave also nicely under multiplication. For reference and background information on the distributions discussed, cf.~\cite{NS,VDN}.

Let $\delta_a$ be a {\bf Dirac} distribution supported at $a\in\R$.
Its Cauchy transform is 
$$
G_{\delta_a}(z)=\int_{\R}\frac{1}{z-t}\delta_a(t)dt=\frac{1}{z-a}~,
$$
and
$$
\frac{1}{{R}_V(z)+(1/z)-a}=z\quad\Leftrightarrow\quad {R}_V(z)=a.
$$

\begin{prop}
Let $a,b\in\R$, and $\delta_a,\delta_b$ the corresponding point masses supported at $a$ and $b$, respectively. Then Dirac distributions are freely Hadamard multiplied by multiplying their supports, i.e. 
$$
\delta_a\boxdot\delta_b=\delta_{ab}~.
$$
\end{prop}

For $a,b\in\R$, $r,s>0$ the {\bf semicircle law} centred at $a$, of radius $r$, is defined as the distribution $\gamma_{a,r}:\C[z]\rightarrow\C$ given by
$$
\gamma_{a,r}(z^n):=\frac{2}{\pi r^2}\int_{a-r}^{a+r} t^n\sqrt{r^2-(t-a)^2}\,dt\qquad \forall n\in\N.
$$
It is determined, as is the Gaussian law, by its first and second moment, i.e., by $\gamma_{a,r}(z)=a$ and $\gamma_{a,r}(z^2)=a^2+r^2/4$. Its $R_V$-transform equals
$$
\label{Rsemicircle}
{R}_{\gamma_{a,r}}(z)=a+\frac{r^2}{4}z~.
$$

\begin{prop}
The family of semicircular distributions is closed both under additive free $\boxplus_V$, and point-wise (Hadamard) free multiplicative $\boxdot$ convolution, i.e. 
\begin{eqnarray*}
\gamma_{a,r}\boxplus_V\gamma_{b,s} & = & \gamma_{a+b,\sqrt{r^2+s^2}} \\
\gamma_{a,r}\boxdot\gamma_{b,s} & = &  \gamma_{ab, rs/2}
\end{eqnarray*}  
\end{prop} 

The {\bf free compound Poisson distribution with rate $\lambda\geq0$ and jump size $\alpha\in\R$} is the limit in distribution for $N\rightarrow\infty$ of
$$
\nu_N:=\nu_{N,\lambda,\alpha}:=\left(\left(1-\frac{\lambda}{N}\right)\delta_0+\frac{\lambda}{N}\delta_{\alpha}\right)^{\boxplus N}
$$ 
with $\boxplus N$ being $n$-fold additive free self-convolution. The free cumulants $(\kappa_n)_{n\in\N^{\times}}$ are given by $\kappa_n=\lambda\alpha^n$ for $n\in\N^{\times}$.

The ${R}_V$-transform of the limit $\nu_{\infty}$ is
$$
{R}_V(\nu_{\infty})(z)=\lambda\alpha\frac{1}{1-\alpha z}=\sum_{n=0}^{\infty}\lambda\alpha^{n+1}z^n=\lambda\alpha+\lambda\alpha^2 z+\lambda\alpha^3z^2+\cdots.
$$

\begin{prop}
The large $N$-limit of the free Poisson distribution $\nu_{\infty,\lambda,\alpha}$ with rate $\lambda=1$ and jump size $\alpha=1$, corresponds to the {\bf unit} with respect to $\boxdot$ multiplication in $\C[[z]]$, i.e. to $\sum_{n\in\N}z^n=1+z+z^2+\dots~.$  
\end{prop}

\section{Free $k$-probability theory}
Let $k$ be a commutative ring with unit $1_k$ and $R\in\mathbf{cAlg}_k$. For $s\in\N^{\times}$ consider the alphabet $[s]:=\{1,\dots,s\}$ with words $w$. The set of all finite words over $[s]$, including the empty word $\emptyset$, is denote by $[s]^*$, and set $[s]^*_+:=[s]^*\setminus\emptyset$.  

\begin{df}
A {\bf non-commutative $R$-valued $k$-probability space} (non-commutative $R-k$-probability space) consists of  a pair $(\mathcal{A},\phi)$, with $\mathcal{A}\in\mathbf{Alg}_k$ and $\phi$ a {\em fixed} $k$-linear functional with values in $R\in\mathbf{cAlg}_k$ such that  
$
\phi(1_{\mathcal{A}})=1_R.
$\\
The elements $a\in \mathcal{A}$ are called  {\bf non-commutative random variables}. 
\end{df}
\begin{df}
The {\bf $s$-dimensional distribution} or {\bf law} of $\phi$, with values in $R$, is the map
$$
\mu:\mathcal{A}^s\rightarrow \Hom_{k,1}(k\langle z_1,\dots, z_s\rangle,R),
$$
such that for any $s$-tuple $\underline{a}:=(a_1,\dots, a_s)$ we have
\begin{eqnarray*}
\label{s-distribution}
\mu_{(a_1,\dots, a_s)}&:&k\langle z_1,\dots, z_s\rangle\rightarrow R\\\nonumber
& & z_w\mapsto \mu_{(a_1,\dots, a_s)}(\underline{a}_w):=\phi(\underline{a}_w)\qquad  \forall w\in[s]^*.
\end{eqnarray*}
Then $\mu_{(a_1,\dots, a_n)}$ is called the {\bf $R$-distribution} of the $s$-tupel $(a_1,\dots, a_s)$. 
\end{df}

\section{Combinatorial Freeness}
The notion of freeness is modified in the general setting.
\begin{df}
Let $(\mathcal{A},\phi)$ be a non-commutative $R-k$-probability space. The subsets $M_1,\dots, M_n\subset\mathcal{A}$ are called {\bf combinatorially free} if they have vanishing mixed free cumulants.
\end{df}
\begin{prop} 
\begin{enumerate}
\item Let $\mathcal{A}_1,\dots, \mathcal{A}_n\subset\mathcal{A}$ be {\em combinatorially free} unital subalgebras in a non-commutative $R-k$-probability space $(\mathcal{A},\phi)$. Then they are also {\em classically free}, i.e., for all $m\geq 1$
$
\phi(a_1\cdots a_m)=0
$
whenever $a_{1}\in\mathcal{A}_{i_1},\dots, a_{m}\in\mathcal{A}_{i_m}$ with $i_1\neq i_2,\dots ,i_{m-1}\neq i_m$ and $\phi(a_i)=0$
\item For any $k$-valued non-commutative $k$-probability space, i.e. $\phi:\mathcal{A}\rightarrow k$, classical and combinatorial freeness are equivalent.
\end{enumerate}
\end{prop}

The essential tool in the proofs is the so-called ``Centering  trick": 

$
a-\underbrace{\phi(a)}_{\in k}1_{\mathcal{A}}
$
which satisfy $\phi(a-\phi(a)1_{\mathcal{A}})=\phi(a)-\phi(a)1_k=0$ as $\phi$ is $k$-linear and unital. 

\section{The combinatorial $\mathcal{R}$-transform}
\begin{df}
Let $\underline{a}=(a_1,\dots,a_s)\in\mathcal{A}^s$, $s\geq 1$. The {\bf $\mathcal{R}$-transform} with values in $R\in\mathbf{cAlg}_k$, is the map 
$
\mathcal{R}:\mathcal{A}^s\rightarrow R_+\langle\langle z_1,\dots, z_s\rangle\rangle
$
which assigns to every $\underline{a}$ the formal power series $\mathcal{R}(\underline{a})$ in $s$ non-commuting variables $\{z_1,\dots, z_s\}$, given by  
\begin{eqnarray*}
\mathcal{R}(\underline{a}):=\mathcal{R}_{\underline{a}}(z_1,\dots,z_s)&:=&\sum^{\infty}_{n=1}\sum^s_{i_1,\dots, i_n=1}\kappa_n(a_{i_1},\dots,a_{i_n})z_{i_1}\cdots z_{i_n}\\
\nonumber
&=&\sum_{|w|\geq1}\kappa_{|w|}(\underline{a}_w) z_w,
\end{eqnarray*} 
where $\kappa_n$ denotes the free cumulants. 
\end{df}

\section{Addition and multiplication of combinatorially free random variables}
For $\underline{a}=(a_1,\dots,a_n), \underline{b}=(b_1,\dots,b_n)\in \mathcal{A}^n$ and $\lambda\in k$ we have
\begin{eqnarray*}
\underline{a}+\underline{b} & := & (a_1+b_1,\dots,a_n+b_n), \\
\underline{a}\star\underline{b} & :=& (a_1\cdot b_1,\dots,a_n\cdot b_n),\\
\lambda\cdot\underline{a} & :=& (\lambda\cdot a_1,\dots,\lambda\cdot a_n).
\end{eqnarray*}
In general, the moment map does not preserve any of the above algebraic structures, e.g. in the additive case we would have
$$
\mathcal{M}(\underline{a}+\underline{b})\neq\mathcal{M}(\underline{a})+\mathcal{M}(\underline{b}),
$$
and in the multiplicative case
$$
\mathcal{M}(\underline{a}\star\underline{b})\neq\mathcal{M}(\underline{a})\star\mathcal{M}(\underline{b}).
$$

\begin{prop}
Let $(\mathcal{A},\phi)$ be a non-commutative $R-k$-probability space, and let $\underline{a}=(a_1,\dots,a_s), \underline{b}=(b_1,\dots,b_s)\in\mathcal{A}^s$ be such that $\{a_1,\dots,a_s\}$ and $\{b_1,\dots,b_s\}$ are combinatorially free. Then
there exist two binary operations  $\boxplus_V$ and $\boxtimes_V$ on $R_+\langle\langle z_1,\dots, z_s\rangle\rangle$
such that the
$\mathcal{M}$-transform defines algebraic morphisms
$$
\mathcal{M}:\mathcal{A}^s\rightarrow R_+\langle\langle z_1,\dots, z_s\rangle\rangle,
$$ 
viz., we have
\begin{itemize}
\item in the {\bf additive} case 
$$
\mathcal{M}(\underline{a}+\underline{b})=\mathcal{M}(\underline{a})\boxplus_V\mathcal{M}(\underline{b})
$$
\item in the {\bf multiplicative} case
$$
\mathcal{M}(\underline{a}\star\underline{b})=\mathcal{M}(\underline{a})\boxtimes_V\mathcal{M}(\underline{b})
$$
\end{itemize}
\end{prop}

\section{Boxed convolution}
Let $\pi$ denote a non-crossing partition and $K(\pi)$ its Kreweras complement.
\begin{df}
\label{boxconv_df}
Let $R\in\mathbf{cAlg}_k$. The {\bf boxed convolution} $\boxtimes$, is a binary operation 
\begin{eqnarray*}
\label{boxconv}
R_+\langle\langle x_1,\dots, x_s\rangle\rangle\times R_+\langle\langle x_1,\dots, x_s\rangle\rangle&\rightarrow& R_+\langle\langle x_1,\dots, x_s\rangle\rangle,\\\nonumber
(f,g)&\mapsto&f\boxtimes g,
\end{eqnarray*}
which for every word $w=(i_1,\dots, i_n)$ and integer $n\geq1$ satisfies:
$$
X_{w}(f\boxtimes g)=\sum_{\pi\in\operatorname{NC}(|w|)} X_{w,\pi}(f)\cdot_R X_{w,K(\pi)}(g).
$$
\end{df}
\begin{rem}
The operation $\boxtimes$ is completely described by {\bf universal polynomials with $\Z$-coefficients}, determined by the underlying combinatorics. 
The calculation of a coefficient of degree $|w|$ depends only on the coefficients of the monomials of degree~$\leq|w|$.
\end{rem}

\section{Moment-cumulant formula}
For any $\underline{a}\in\mathcal{A}^s$, we have: 
$
\mathcal{R}(\underline{a})\boxtimes\operatorname{Zeta}=\mathcal{M}(\underline{a}),
$
i.e.,
\begin{equation*}
\label{RSmu}
\begin{xy}
  \xymatrix{
 &  \mathcal{A}^s\ar[dl]_{\text{$\mathcal{R}$-transform}\quad} \ar[dr]^{\quad\text{$\mathcal{M}$-transform}}   &\\
                  R_+\langle\langle z_1,\dots, z_s\rangle\rangle\ar[rr]^{\boxtimes\operatorname{Zeta}}           &    & R_+\langle\langle z_1,\dots, z_s\rangle\rangle
  }
\end{xy}
\end{equation*}

\begin{prop}
The relation between $+$ and $\boxplus_V$, and between $\boxtimes$ and $\boxtimes_V$, respectively, is given by 
$$
\boxplus_V=\left(\bullet_1\boxtimes\operatorname{Moeb}_s+\bullet_2\boxtimes\operatorname{Moeb}_s\right)\boxtimes\operatorname{Zeta}_s,
$$
and
$$
\boxtimes_V=\bullet_1\boxtimes\operatorname{Moeb}_s\boxtimes\bullet_2.
$$
with $\bullet_i$, $i=1,2$, denoting the first and the second argument, respectively.
\end{prop}

\section{The $k$-functors}
Let $s\in\N^{\times}$, $R\in\mathbf{cAlg}_k$.
\begin{eqnarray*}
\mathfrak{M}^s(R)& := & \big\{\sum_{|w|\geq 1}\alpha_w z_w~|~\alpha_w\in R, w\in [s]^*\big\}\cong R^{[s]_+^*},\\
\mathfrak{G}^s(R) & := & \{(r_1,\dots, r_s,\dots, r_w,\dots)~|~r_i\in R^{\times}, i\in[s], r_w\in R, |w|\geq 2\},\\
\mathfrak{G}^s_+(R) & := & \{(1_1,\dots, 1_s,\dots, r_w,\dots)~|~r_w\in R, |w|\geq 2\},
\end{eqnarray*}
and the finite-dimensional restrictions of $\mathfrak{M}^s(R)$, $\mathfrak{G}^s(R)$ and $\mathfrak{G}^s_+(R)$, respectively, for $n\geq1$
\begin{eqnarray*}
(\mathfrak{M}^s)_n(R)&:=& \big\{\sum_{1\leq|w|\leq n }\alpha_w z_w\big\},\\
(\mathfrak{G}^s)_n(R) & := & \big\{\sum_{1\leq|w|\leq n }\alpha_w z_w~| \alpha_i\in R^{\times}\quad\text{for $i=1,\dots, s$}\big\},\\
(\mathfrak{G}^s_+)_n(R) & := & \big\{\sum_{1\leq|w|\leq n }\alpha_w z_w~|~\alpha_i=1_R\quad\text{for $i=1,\dots, s$}\big\}.
\end{eqnarray*} 

\section{Affine group schemes}
\begin{prop}
\label{Group_prop}
For all $s\in\N^{\times}$, $R\in\mathbf{cAlg}_k$ and the binary operation $\boxtimes$, the following hold:
\begin{enumerate}
\item $(\mathfrak{M}^s(R),+)$, with component-wise addition, is an abelian group, and it has as additive neutral element, $0_{\mathfrak{M}^s(R)}$, corresponding to the series with all coefficients being zero.
\item $(\mathfrak{M}^s(R),\boxtimes)$ is an associative multiplicative monoid, with unit 
$$
1_{\mathfrak{M}^s(R)}=1_Rz_1+\cdots+1_R z_s.
$$
\item $(\mathfrak{M}^s(R),+,\boxtimes)$, is an associative and {\bf not distributive} unital ring, which for $s\geq2$ is also not commutative.
\item $\mathfrak{G}^s(R)$ is a group with neutral element
$
1_{\mathfrak{G}^s(R)}=1_Rz_1+\cdots+1_R z_s.
$\\
For $s=1$ it is abelian and for $s\geq2$ non-abelian.
\end{enumerate}
\end{prop}

\begin{prop}
For all $s\in\N^{\times}$, $R\in\mathbf{cAlg}_k$, and the binary operation $\boxtimes$, the following holds:
\begin{enumerate}
\item $\mathfrak{G}^s(R)$ is a group, which is the {\bf semi-direct product} of the {\bf normal subgroup} $\mathfrak{G}^s_+(R)$ and the {\bf $s$-torus}
$(\mathfrak{G}^s)_1(R)\cong\mathbb{G}_m^s(R)$, i.e.,
$$
\mathfrak{G}^s(R)=\mathfrak{G}^s_1(R)\ltimes (\mathfrak{G}^s_+)(R)~,
$$
or equivalently, we have the exact sequence
\[
\begin{xy}
  \xymatrix{
                             0_R\ar[r]&\mathbb{G}_m^s(R)\ar[r]^{\iota} &\mathfrak{G}^s(R)\ar[r]^p &\mathfrak{G}^s_+(R)\ar[r] &1_R~.
               }
\end{xy}
\]
\item $\mathfrak{G^s}(R)$ and $\mathfrak{G}^s_+(R)$ are groups, filtered in ascending order by the subgroups $(\mathfrak{G}^s)_n(R)$ and $(\mathfrak{G}^s_+)_n(R)$, respectively.
\item $\mathfrak{G}^s(R)$ and $\mathfrak{G}^s_+(R)$ are {\em projective limits of finite dimensional groups}, i.e. 
$$
\mathfrak{G}^s(R)=\varprojlim_{n\in\N} (\mathfrak{G}^s)_n(R)\qquad\text{and}\quad\mathfrak{G}^s_+(R)=\varprojlim_{n\in\N} (\mathfrak{G}^s_+)_n(R)~.
$$
\end{enumerate}
\end{prop}

\section{Lie Groups and Lie Algebras in characteristic 0}
\begin{thm} 
\label{main-iso}
Let $R$ be a $\Q$-algebra. Then the formal groups $\mathbb{G}_a^{\N}(R)$ and $(\mathfrak{G}^1_+(R),\boxtimes)$ are {\em isomorphic}.
\end{thm}
\begin{proof}
\begin{itemize}
\item This follows from {\bf Lazard's Theorem}: ``Over $\Q$ every commutative formal group law $F$ is isomorphic to the $\dim(F)$-dimensional additive group law".
\item $(\mathfrak{G}^1_+(R),\boxtimes)$ defines an infinite commutative formal group law.
\end{itemize}
\end{proof}
\begin{rem}
\begin{itemize}
\item Our $\operatorname{LOG}$ in Thm.~\ref{LOG} is the concrete form of the above isomorphism.
\item The above theorem answers one of the question we started with.
\item For $s\geq 2$, $(\mathfrak{G}^s_+,\boxtimes)$ is not isomorphic to the additive group, as it is not commutative. 
\end{itemize} 
\end{rem}

\section{The co-ordinate algebras}
The $k$-functors are group valued and representable. By the Yoneda Lemma this gives rise to Hopf algebras, which we described next. 
\begin{thm}
\label{first_main_thm}
Let $s\in\N^{\times}$ and $k$ be a ring. For the formal group schemes we have
\begin{itemize}
\item $\mathfrak{G}^s$ is represented by
$$
k[X^{\pm 1}_1,\dots, X^{\pm 1}_s, X_w:|w|\geq 2];
$$
and it is not connected.
\item $(\mathfrak{G}^s)_1$ is represented by  
$$ k[X_1,X^{-1}_1,\dots,X_s,X^{-1}_s],$$
consisting of {\em group-like} elements.
\item $\mathfrak{G}^s_+$ is represented by the graded and connected Hopf algebra
$$
k[\bar{X}_w:|w|\geq 2]
$$
which for $s=1$ is co-commutative. 
\end{itemize}
\end{thm}

The co-structure is given by:
\begin{itemize}
\item The co-product: $\Delta X_w(f,g):=X_w(f\boxtimes g)$, 
$$
\Delta X_w=\sum_{\pi\in\operatorname{NC}(|w|)} X_{w,\pi}\otimes X_{w,K(\pi)},
$$
where $X_{w,\pi}$ and $X_{w,K(\pi)}$ are as above.
\item For the co-unit: $\varepsilon(X_w):=X_w(1_{(\mathfrak{G}^s)_n})$ we get
\begin{equation}
\varepsilon(X_w)=\begin{cases}
      & 1\quad \text{for $w=(i), i\in\{1,\dots,s\}$}, \\
      &  0\quad \text{for $|w|\geq2$}.
\end{cases}
\end{equation}
\end{itemize}
This defines on $R[X_w| 1\leq|w|\leq n]$ the structure of a {\bf bi-algebra}.  
\begin{rem}
The elements $X_i$, $i\in\{1,\dots s\}$, are {\bf group-like}, i.e., they satisfy 
$$
\Delta(X_i)=X_i\otimes X_i\qquad\text{and}\qquad\varepsilon(X_i)=1.
$$
\end{rem}
The full structure is obtained with
\begin{prop}[Antipode of the boxed convolution]
Let $w=(i_1\dots i_n)\in[s]^*$ with $n\geq 2$. The antipode corresponding to the boxed convolution $\boxtimes$, is given by the {\em recursive} relation
\begin{equation}
\label{antipode}
S(X_w)=S(X_{(i_1\dots i_n)})=-(X^{-1}_{i_1}\cdots X^{-1}_{i_n})\cdot\sum_{\pi\neq 0_{n}}X_{w,\pi}\cdot S(X_{w,K(\pi)}).
\end{equation}
\end{prop}

\section{Some types of matrices}
In the remaining sections we assume $k$ to be a field.
\begin{itemize}
\item The set of {\bf diagonal matrices}  
$$
\mathbb{T}_n=\{\operatorname{diag}(t_1,\dots,t_n)~|~t_i\in k^{\times}\}
$$ 
in $\operatorname{GL}_n(k)$ is called the {\bf torus}, 
\item the set of {\bf upper-triangular} matrices 
$$
\mathbb{B}_n(k):=\{(a_{ij})_{i,j}\in\operatorname{GL}_n(k)~ | \text{$a_{ij}=0$ for $i>j$}\}
$$ known as  the {\bf Borel subgroups}. 
\end{itemize}
Let 
$$
\mathbb{B}_{\infty}:=\varprojlim \mathbb{B}_n(k)\quad\text{and}\quad\mathbb{U}_{\infty}:=\varprojlim \mathbb{U}_n(k)
$$
be the corresponding {\bf projective limits}.

\begin{df}
Let $A$ be a square matrix $A$ and $E$ the unit matrix. Then, $A$ is called {\bf unipotent} if $A-E$ is {\bf nilpotent},  i.e., there exists an integer $n\geq1$ such that $(A-E)^n=0$.\\
An element $r$ in a ring $R$ with unit $1_R$, is called unipotent, if $(r-1_R)$ is nilpotent, i.e., $\exists n\geq 1$, $(r-1_R)^n=0_R$.
\end{df}
The the {upper-triangular} matrices with constant diagonal $1$ are unipotent
$$
\mathbb{U}_n(k):=\left(\begin{array}{ccccc}1 & * & \dots  & * \\0 & 1 & \ddots & \vdots  \\ \vdots& \ddots & \ddots & *  \\ 0& \cdots &  0& 1 \end{array}\right)
$$
with $*$ denoting arbitrary entries from $k$. 

The $n$-dimensional {\bf additive group} $\mathbb{G}^n_a(k)$ is unipotent and is faithfully represented by 
$$
\left(\begin{array}{cccccc}1 & x_1 & \cdots  & \cdots  &x_n\\0 & 1 & 0 & \dots  &0\\ \vdots& \ddots & \ddots & \ddots &\vdots\\ 0& \cdots &  0& 1 &0\\
0&0&\cdots&0&1\end{array}\right)
$$
with $x_i\in k$, $i=1,\dots, n$.

\section{Faithful representations}
\begin{prop} 
The groups $((\mathfrak{G}^s)_n,\boxtimes)$ and $((\mathfrak{G}^s_+)_n,\boxtimes)$ are {\bf faithfully representable} as closed subgroups of some $\operatorname{GL}_N$, $N=N(n,s)$ and $(\mathfrak{G}^s,\boxtimes)$ and $(\mathfrak{G}^s_+,\boxtimes)$ are pro-algebraic groups, i.e. projective limits of algebraic groups. 
\end{prop}
\begin{thm}
\label{Higher-S-trafo}
Let $k$  be a field and $s,n\in\N^{\times}$. There exist {\bf faithful} representations $\rho_{n,s}$, for
\begin{itemize}
\item $((\mathfrak{G}^s)_n(k),\boxtimes)$ as closed sub-groups of the {\bf Borel groups} $\mathbb{B}_N(k)$, $N=N(n,s)$,
\item $((\mathfrak{G}^s_+)_n(k),\boxtimes)$ as a closed sub-groups of the {\bf unipotent groups} $\mathbb{U}_N(k)$, $N=N(n,s)$,
\item $((\mathfrak{G}^s)_1(k),\boxtimes)$ as a group of {\bf multiplicative type} $\mathbb{T}_s(k)$.
\end{itemize}
\end{thm}

\section{The general $S$-transform}
\begin{df}[$S$-transform]
The minimal faithful representation $\rho_k:(G^s_+(k),\boxtimes)\rightarrow\mathbb{B}_{\infty}(k)$ is called the {\bf $s$-dimensional $\mathcal{S}$-transform}. 

If $(\mathcal{A},\phi)$ is a non-commutative $k$-probability space and $\underline{a}=(a_1,\dots,a_s)\in\mathcal{A}^s$ then the $\mathcal{S}$-transform of $\underline{a}$ is defined as
\begin{equation}
\label{}
\mathcal{S}_k(\underline{a}):=\rho_k(\mathcal{R}_k(\underline{a}))~.
\end{equation}
\end{df}
\begin{thm}
Let $(\mathcal{A},\phi)$ be a non-commutative $k$-probability space and  $\underline{a}=(a_1,\dots,a_s), \underline{b}=(b_1,\dots,b_s)\in\mathcal{A}^s$ such that $\{a_1,\dots,a_s\}$ and $\{b_1,\dots,b_s\}$ are combinatorially free. 
Then the $\mathcal{S}$-transform is a faithful morphism
$$
\mathcal{S}_k:\mathcal{A}^s\rightarrow \varprojlim\mathbb{B}_N(k)
$$
which satisfies
\begin{equation}
\label{R_trafo_mult}
\mathcal{S}_k(\underline{a}\star\underline{b})=\mathcal{S}_k(\underline{a})\cdot\mathcal{S}_k(\underline{b})
\end{equation}
\end{thm}
\begin{rem}
An equivalent description can be given in terms of moment series, by using the moment-cumulant formula.
\end{rem}

Authors addresses:

Roland Friedrich, 53115 Bonn, Germany\\ rolandf@mathematik.hu-berlin.de

John McKay, Dept. Mathematics, Concordia University, Montreal, Canada H3G 1M8\\
mac@mathstat.concordia.ca
\end{document}